\documentclass{ws-ijwmip}

\usepackage{comment}
\usepackage{IEEEtrantools,amsmath,amssymb}
\usepackage{ifthen}
\usepackage{url}

\usepackage{tikz}

\usepackage[OT2, T1]{fontenc}
\usepackage[russian, USenglish]{babel}

\newcommand \indexkeyword[1]{\emph{#1}\index{#1}}

\newcommand \eqnlabel[1]{\yesnumber\label{eqn:#1}}    
\newcommand \dfnlabel[1]{\label{dfn:#1}}
\newcommand \thmlabel[1]{\label{thm:#1}}
\newcommand \lemlabel[1]{\label{lem:#1}}

\newcommand \seclabel[1]{\label{sec:#1}}
\newcommand \figlabel[1]{\label{fig:#1}}

\newcommand \eqnref[1]{(\ref{eqn:#1})}      

\newcommand \thmref[1]{Theorem~\ref{thm:#1}}    
\newcommand \lemref[1]{Lemma~\ref{lem:#1}}

\newcommand \figref[1]{Figure~\ref{fig:#1}}     

\ifthenelse{\isodd{1}}{
  
  \newcommand\translater[2]{T^{#1}#2}
  \newcommand\translaterone[1]{T#1}
  
  \newcommand\shrink[2]{S_{#1}#2}
}{
  
  \newcommand\translater[2]{#2 \rightarrow #1}
  \newcommand\translaterone[1]{\translater{1}{#1}}
  
  \newcommand\shrink[2]{#2 \downarrow #1}
}

\newcommand \polyfunc[1]{\widehat{#1}}

\newcommand \mapelement{\mapsto}

\newcommand \createfunc[2]{\left(#1\mapelement#2\right)}

\DeclareMathOperator*{\dif}{d}
\DeclareMathOperator*{\diagid}{diag}

\newcommand \diag[1]{\diagid \left(#1\right)}   

\newcommand \mset[1]{\left\{#1\right\}} 

\newcommand \field[1]{\mathbb{#1}}
\newcommand \N{\field{N}}

\newcommand \Z{\field{Z}}
\newcommand \R{\field{R}}

\newcommand \summablesequences[2]{\ell_{#1}\left(#2\right)}

\newcommand \il{\nu}
\newcommand \ir{\kappa}

\newcommand \person[1]{{\sc #1}}

\newcommand \Daubechies{\person{Daubechies}}
\newcommand \Dirac{\person{Dirac}}

\newcommand \Fourier{\person{Fourier}}

\newcommand \Laurent{\person{Laurent}}

\begin{document}

\newcommand\mytitle{How to refine polynomial functions}

\markboth{Henning Thielemann}{\mytitle}

\catchline{}{}{}{}{}

\title{\mytitle}

\author{Henning Thielemann}

\address{Institut f\"{u}r Informatik \\
Martin-Luther-Universit\"{a}t Halle-Wittenberg \\
06110 Halle \\
Germany \\
henning.thielemann@informatik.uni-halle.de}

\maketitle

\def\extended#1{Extended~#1\par}

\begin{history}
\received{28.11.2010}
\accepted{26.03.2011}
\revised{13.07.2011}
\extended{23.11.2011}
\end{history}

\begin{abstract}
Research on refinable functions in wavelet theory
is mostly focused to localized functions.
However it is known, that polynomial functions are refinable, too.
In our paper we investigate on conversions between
refinement masks and polynomials and their uniqueness.
\end{abstract}



\graphicspath{{figures/}}

\section{Introduction}

\newcommand\drawframe[6]{
  \def\left{#1}
  \def\right{#2}
  \def\xticks{#3}
  \def\bottom{#4}
  \def\top{#5}
  \def\yticks{#6}
  \draw (\left,\bottom) rectangle (\right,\top);
  \foreach \x in \xticks {
     \draw (\x,\bottom) node[below] {\x}
                     -- (\x,0.95*\bottom+0.05*\top);
     \draw (\x,\top) -- (\x,0.05*\bottom+0.95*\top);
  }
  \foreach \y in \yticks {
     \draw (\left, \y) node[left] {\y}
                       -- (0.98*\left+0.02*\right,\y);
     \draw (\right,\y) -- (0.02*\left+0.98*\right,\y);
  }

}

\indexkeyword{Refinable functions} are functions
that are in a sense self-similar:
If you add shrunken translates of a refinable function in a weighted way,
then you obtain that refinable function again.
For instance, see \figref{refinement} for how a quadratic B-spline
can be decomposed into four small B-splines
and how the so called \Daubechies-2 generator function
is decomposed into four small variants of itself.
\begin{figure}
\begin{tabular}{cc}
\begin{tikzpicture}[domain=0:3,xscale=1.5,yscale=3]
  \drawframe{-0.2}{3.2}{0,0.5,...,3} {-0.1}{0.9}{0.0,0.2,0.4,0.6,0.8}
  \draw plot file {figures/RefinementBSpline.csv};
  \draw[dash pattern=on 2pt off 1pt] plot file {figures/RefinementBSpline0.csv};
  \draw[dash pattern=on 1pt off 1pt] plot file {figures/RefinementBSpline1.csv};
  \draw[dash pattern=on 2pt off 1pt] plot file {figures/RefinementBSpline2.csv};
  \draw[dash pattern=on 1pt off 1pt] plot file {figures/RefinementBSpline3.csv};
\end{tikzpicture}
&
\begin{tikzpicture}[domain=0:3,xscale=1.5,yscale=3/2.6]
  \drawframe{-0.2}{3.2}{0,0.5,...,3} {-0.6}{2}{-0.5,0,...,2}
  \draw plot file {figures/RefinementDaubechies.csv};
  \draw[dash pattern=on 2pt off 1pt] plot file {figures/RefinementDaubechies0.csv};
  \draw[dash pattern=on 1pt off 1pt] plot file {figures/RefinementDaubechies1.csv};
  \draw[dash pattern=on 2pt off 1pt] plot file {figures/RefinementDaubechies2.csv};
  \draw[dash pattern=on 1pt off 1pt] plot file {figures/RefinementDaubechies3.csv};
\end{tikzpicture}
\\
Quadratic B-spline & \Daubechies-2 wavelet generator
\end{tabular}
\caption{%
Refinement of a quadratic B-spline and the orthogonal \Daubechies-2 generator.
The black line is the refinable function
that is composed from shrunken, translated and weighted
versions of itself, displayed with dashed lines.
}
\figlabel{refinement}
\end{figure}
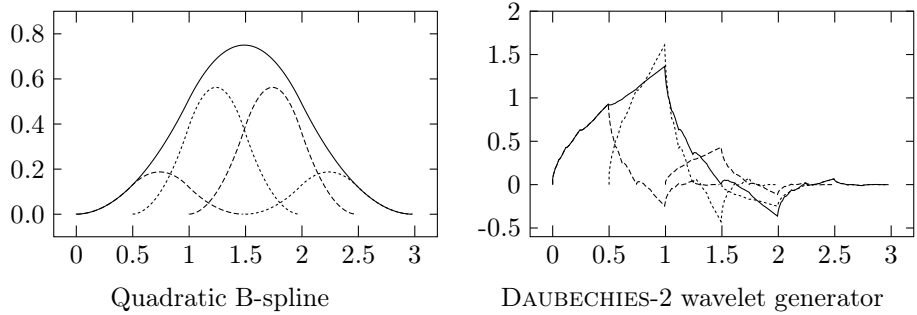

All B-splines with successive integral nodes are refinable,
but there are many more refinable functions
that did not have names before the rise of the theory of refinable functions.
In fact we can derive a refinable function from the weights
of the linear combination in the refinement under some conditions.

Refinable functions were introduced
in order to develop a theory of real wavelet functions
that complements the discrete sub-band coding theory.%
\cite{strang1997filterbanks}
Following the requirements of wavelet applications,
existing literature on wavelets focuses on refinable functions
that are $\mathcal{L}_2$-integrable
and thus have a well-defined \Fourier{} transform,
are localized (finite variance) or even better of compact support.
It is already known,
that polynomial functions are refinable as well.%
\cite{gustafson2006refine}
In this paper we want to explore in detail
the connection between polynomials and the respective weights for refinement.

Our results can be summarized as follows:
\begin{itemize}
\item Masks that sum up to a negative power of two
refine polynomials that are uniquely defined up to constant factors.
Other masks are not associated with a polynomial.
(\thmref{mask-to-polynomial})
\item For every polynomial there are infinitely many refinement masks,
and these refinement masks can be characterized in a simple form.
(\thmref{polynomial-to-mask} and \thmref{polynomial-to-all-masks})
\item There is a simple iterative algorithm for approximating
a polynomial that is associated with a mask.
(\thmref{cascade-algorithm})
\end{itemize}

\section{Main Work}

\subsection{Basics}

We start with a precise definition of refinable functions.
\begin{definition}[Refinable function]
\dfnlabel{mask}
The vector~$m$ with $m\in\R^{\Z}$
and a finite number of non-zero entries
($m\in\summablesequences{0}{\Z}$)
is called a \indexkeyword{refinement mask} for the function~$\varphi$
if
\begin{IEEEeqnarray*}{rCl}
\varphi(t)
 &=& 2 \cdot \sum_{j\in\Z} m_j \cdot \varphi(2\cdot t-j)
  \eqnlabel{refinement}
\end{IEEEeqnarray*}
holds. Vice versa the function~$\varphi$ is called
\indexkeyword{refinable} with respect to the mask~$m$.
\end{definition}

The factor~$2$ before the sum is chosen,
such that the following law (\lemref{refinable-convolution})
about convolutions holds.
Unfortunately this enforces adding or subtracting $1$ here and there
in some of the other theorems.
There seems to be no convention that asserts overall simplicity.
\begin{definition}[Convolution]
\dfnlabel{convolution}
For sequences $h$ and $g$ the convolution is defined by
\[ (h * g)_k = \sum_{j\in\Z} h_{j} \cdot g_{k-j} \]
and for real functions the convolution is defined by
\[
(\varphi * \psi)(t)
   = \int_{\R} \varphi(\tau) \cdot \psi(t-\tau) \dif \tau
\qquad.
\]
\end{definition}

\begin{lemma}
\lemlabel{refinable-convolution}
If $\varphi$ is refinable with respect to $h$
and $\psi$ is refinable with respect to $g$,
then $\varphi * \psi$ is refinable with respect to $h * g$.
\end{lemma}
For a proof see \refcite{thielemann2006matchedwavelets}.
%
For the proof of our theorems
we need two further lemmas about differentiation and integration.
\begin{lemma}
\lemlabel{derivative}
If the function~$\varphi$ is refinable with respect to mask~$m$,
then its derivative $\varphi'$ is refinable with respect to mask~$2\cdot m$.
\end{lemma}
\begin{proof}
Derive both sides of the refinement equation with respect to $t$.
\begin{IEEEeqnarray*}{rCl}
\varphi(t)
 &=& 2 \cdot \sum_{j\in\Z} m_j \cdot \varphi(2\cdot t-j) \\
\varphi'(t)
 &=& 2 \cdot \sum_{j\in\Z} m_j \cdot 2\cdot\varphi'(2\cdot t-j)
\end{IEEEeqnarray*}
\end{proof}

\begin{lemma}
\lemlabel{antiderivative}
If the function~$\varphi$ is refinable with respect to mask~$m$
and there is an antiderivative~$\Phi$ with
$\Phi(t) = \int_0^{t}\varphi(\tau)\dif\tau$,
then the antiderivative $\createfunc{t}{\Phi(t) + c}$
is refinable with respect to mask~$\frac{1}{2}\cdot m$
where the constant~$c$ must be chosen as follows:
\begin{itemize}
\item
For $\sum_j m_j \ne 1$
it must be
$c = \frac{\sum_j m_j \cdot \Phi(-j)}{1 - \sum_j m_j}$.
\item
For $\sum_j m_j = 1$
and $\sum_j m_j \cdot \Phi(-j) = 0$
the constant~$c$ can be chosen arbitrarily.
\item
For $\sum_j m_j = 1$
and $\sum_j m_j \cdot \Phi(-j) \ne 0$
there is no valid value for the constant~$c$.
\end{itemize}
These choices for $c$ are the only possible ones.
\end{lemma}
\begin{proof}
We start with the necessary condition; that is,
given that the antiderivative $\createfunc{t}{\Phi(t) + c}$
is refinable with respect to mask $\frac{1}{2}\cdot m$,
what are the possible integration constants?
\begin{IEEEeqnarray*}{rCl}
\Phi(t) + c
 &=& 2 \cdot \sum_{j\in\Z} \frac{1}{2}\cdot m_j
          \cdot \left(\Phi(2\cdot t-j)+c\right) \\
\Phi(t) + c\cdot(1 - \sum_{j\in\Z} m_j)
 &=& \sum_{j\in\Z} m_j \cdot \int_0^{2\cdot t-j}\varphi(\tau)\dif\tau \\
 &=& \sum_{j\in\Z} m_j \cdot \left(\int_{-j}^{2\cdot t-j}\varphi(\tau)\dif\tau + \Phi(-j)\right) \\
 &=& \sum_{j\in\Z} m_j \cdot \left(2\cdot\int_{0}^{t}\varphi(2\cdot \tau-j)\dif\tau + \Phi(-j)\right) \\
 &=& \int_{0}^{t}\left(\sum_{j\in\Z} 2\cdot m_j \cdot \varphi(2\cdot \tau-j)\right)\dif\tau + \sum_{j\in\Z} m_j \cdot \Phi(-j) \\
 &=& \int_{0}^{t} \varphi(\tau)\dif\tau + \sum_{j\in\Z} m_j \cdot \Phi(-j) \\
c\cdot(1 - \sum_{j\in\Z} m_j)
 &=& \sum_{j\in\Z} m_j \cdot \Phi(-j) \\
\end{IEEEeqnarray*}
Proof of the sufficient condition:
By substituting $c$ by the admissible values we verify,
that the antiderivative with that offset
is actually refined by $\frac{1}{2}\cdot m$.
\end{proof}

\begin{remark}
For $\sum_j m_j=1$ and $m$ and $\varphi$ with support,
that is bounded on at least one side,
the second case of \lemref{antiderivative} applies.
This means that all antiderivatives of $\varphi$
irrespective of the integration constant are refinable
with respect to $\frac{1}{2}\cdot m$.

Sketch of the proof:
Without loss of generality let $m$ and $\varphi$ have support,
that is bounded at the left.
The refinement equation implies,
that the bounds of their support are equal.
If their support is entirely on the positive real axis,
then $\forall t\le0 \  \Phi(t)=0$
and further on in all summands of $\sum_j m_j \cdot \Phi(-j)$
at least one factor is zero.
This implies $\sum_j m_j \cdot \Phi(-j) = 0$.

Now, when $\varphi$ is refined by $m$ and $k$ is an integer,
then $\varphi$ translated by $k$ is refinable with respect to $m$ translated by $k$.
This way we can reduce all $m$ and $\varphi$ with bounded support
to ones with support on the positive real axis.

\end{remark}

\subsection{Conversions between polynomials and masks}

If we generalize refinable functions to refinable distributions,
then the \Dirac{} impulse is refined by the mask~$\delta$
with $\delta_j = \begin{cases} 1 &: j=0 \\ 0 &: j\ne 0\end{cases}$
and the $k$-th derivative of the \Dirac{} impulse is refined by $2^k\cdot\delta$.
Vice versa the truncated power function
$\createfunc{t}{t_+^k}$
with $k\in\N$
and $t_+ = \begin{cases} t &: t\ge0 \\ 0 &: t<0\end{cases}$
is refined by $2^{-k-1}\cdot\delta$.
Intuitively said,
truncated power functions are antiderivatives of the \Dirac{} impulse.

Once we are thinking about truncated power functions,
we find that ordinary power functions with natural exponents are also refinable.
Then it is no longer a surprise,
that polynomial functions are refinable, too.
For example $f$ with
$f(t) = 1 + 2\cdot t + t^2$
is refined by the mask $\frac{1}{8}\cdot(3,-3,1)$:
\begin{IEEEeqnarray*}{rCl}
\noalign{$2\cdot\frac{1}{8}\cdot(3\cdot f(2t) - 3\cdot f(2t-1) + 1\cdot f(2t-2))$}
 &=& \frac{1}{4}\cdot
        (3\cdot (1 + 2\cdot 2t + (2t)^2)
       - 3\cdot (1 + 2\cdot (2t-1) + (2t-1)^2) \\
 && \qquad\qquad
       +        (1 + 2\cdot (2t-2) + (2t-2)^2)) \\
 &=& \frac{1}{4}\cdot
        (3\cdot (1 + 4t + 4t^2)
       - 3\cdot 4t^2
       +        (1 - 4t + 4t^2)) \\
 &=& 1 + 2\cdot t + t^2
\quad.
\end{IEEEeqnarray*}

Now that we have an example of a refinable polynomial function,
we like to know how we can find a mask that refines a polynomial function.
Vice versa we want to know a characterization of masks
that refine polynomial functions
and what polynomial functions can be refined by a given mask.

Before we start answering these questions
we would like to stress the difference between
a \indexkeyword{polynomial} and a \indexkeyword{polynomial function}.
\begin{definition}[Polynomial and Polynomial function]
A polynomial~$p$ of degree~$n$ is a vector from $\R^{\{0,\dots,n\}}$.
We need this for the actual computations
and for performing linear algebra.
A polynomial function~$\polyfunc{p}$ is a real function.
The refinement property is a property of real functions.
The connection between polynomial and polynomial function is
\[
\polyfunc{p}(t) = \sum_{k=0}^{n} p_k\cdot t^k
\qquad.
\]
\end{definition}

Our first theorem answers the question,
``What polynomial can be refined by a mask?''
\begin{theorem}
\thmlabel{mask-to-polynomial}
Given a mask~$m$ that sums up to $2^{-n-1}$ for a given natural number~$n$,
there is a polynomial~$p$ of degree~$n$
such that $m$ refines $\polyfunc{p}$.
With the additional condition
of the leading coefficient being $1$,
this polynomial is uniquely determined.
\end{theorem}
\begin{proof}
We show this theorem by induction over $n$.
\begin{itemize}
\item Case $n=0$ \\
We want to show that a mask~$m$ with sum $\frac{1}{2}$
can only refine a constant polynomial.
Thus we assume contrarily that $m$
refines a polynomial with a degree~$d$ greater than zero.
In the refinement relation
\begin{IEEEeqnarray*}{rCl}
\polyfunc{p}(t)
 &=& 2\cdot \sum_{j\in\Z} m_j \cdot \polyfunc{p}(2\cdot t-j)
\end{IEEEeqnarray*}
we only consider the leading coefficient,
that is, the coefficient of $t^d$.
\begin{IEEEeqnarray*}{rCl}
p_d
 &=& 2\cdot \sum_{j\in\Z} m_j \cdot 2^d \cdot p_d \\
 &=& 2^d \cdot p_d
\end{IEEEeqnarray*}
From $d>0$ it follows that $p_d = 0$.

Thus the degree~$d$ must be zero
and by normalization it must be $p_0 = 1$.
We can easily check that this constant polynomial
is actually refined by any mask with sum~$\frac{1}{2}$.

\item Case $n>0$: Induction step \\
The induction hypothesis is that for any mask
with coefficient sum $2^{-n}$
we can determine a refining polynomial of degree~$n-1$,
that is unique when normalized so that the leading coefficient is $1$.
The induction claim is
that for a mask~$m$ with sum $2^{-n-1}$
we have a uniquely determined polynomial of degree~$n$
with leading coefficient~$1$.
We observe that $2\cdot m$ satisfies the premise of the induction hypothesis
and thus there is a polynomial~$q$ of degree~$n-1$
that is refined by $2\cdot m$
and that is unique when normalized.
Since the coefficient sum of $2\cdot m$ is at most $\frac{1}{2}$,
it is different from $1$
and thus the first case of \lemref{antiderivative} applies.
It lets us obtain the $m$-refinable polynomial~$p$ in the following way:
Let $Q$ be the antiderivative polynomial of $q$
where the constant term is zero,
then it is
\begin{IEEEeqnarray*}{r"rCl}
\forall k>0 &
  p_k &=& \frac{Q_k}{Q_n} \\
& p_0 &=& \frac{2}{Q_n\cdot(1 - 2^{-n})} \cdot
             \sum_{j\in\Z} m_j \cdot \polyfunc{Q}(-j)
\qquad.
\end{IEEEeqnarray*}

Now we turn to the question of why $p$ is uniquely determined.
Assume we have two normalized polynomial functions
$\polyfunc{p_0}$ and $\polyfunc{p_1}$
that are both refined by mask~$m$.
Then their derivatives $\polyfunc{p_0}'$ and $\polyfunc{p_1}'$
are refined by mask~$2\cdot m$.
Due to the induction hypothesis
the normalized polynomial functions
of $\polyfunc{p_0}'$ and $\polyfunc{p_1}'$ are equal.
The first case of \lemref{antiderivative} implies
that the antiderivatives with respect to
$\polyfunc{p_0}'$ and $\polyfunc{p_1}'$
have the same integration constant,
and thus $p_0=p_1$.
\end{itemize}
\mbox{} 
\end{proof}

\begin{theorem}
\thmlabel{polynomial-to-mask}
Given a polynomial~$p$ of degree~$n$,
there is a uniquely defined mask~$m$
of support in $\{0,\dots,n\}$ that refines $\polyfunc{p}$.
\end{theorem}

For the proof of that theorem we introduce some matrices.
\begin{definition}
We express shrinking a polynomial by factor $k$ by the matrix~$S_k$.
\begin{IEEEeqnarray*}{rCl}
S_k &\in& \R^{\{0,\dots,n\} \times \{0,\dots,n\}} \\
S_k &=& \diag{1,k,\dots,k^n}
\end{IEEEeqnarray*}
We represent translation of a polynomial by $1$ by the matrix~$T$
and translation of a distance~$i$ by the power $T^i$.
\begin{IEEEeqnarray*}{rCl}
T &\in& \R^{\{0,\dots,n\} \times \{0,\dots,n\}} \\
(T^i)_{j,k} &=&
\begin{cases}
  \binom{k}{j} \cdot (-i)^{k-j}
    &: j \le k \\
  0 &: j > k
\end{cases}
\end{IEEEeqnarray*}
\end{definition}

The proof of \thmref{polynomial-to-mask} follows.
\begin{proof}
We define the matrix~$P$
that consists of translated polynomials as columns.
\[ P = (\translater{0}{p}, \dots, \translater{n}{p}) \]
Now computing $m$ is just a matter of solving
the simultaneous linear equations
\[ p = 2\cdot\shrink{2}{P m} \qquad. \]
We only have to show that $P$ is invertible.
We demonstrate that by doing a kind of LU decomposition,
that also yields an algorithm for actually computing $m$.
Our goal is to transform $P$ into triangular form
by successive subtractions of adjacent columns.
We define
\[ \Delta p = \translaterone{p} - p \]
what satisfies
\[ \Delta (\translater{k}{p}) = \translater{k}{\Delta p} \qquad. \]
In the first step we replace all but the first columns of $P$
by differences, yielding the matrix~$U_1$.
\begin{IEEEeqnarray*}{rCl}
U_1
 &=& ({p},
      {\Delta p},
      \translaterone{\Delta p}, \dots,
      \translater{n-1}{\Delta p}) \\
 &=& P\cdot L_1^{-1}
\end{IEEEeqnarray*}
\[
L_1^{-1} =
\begin{pmatrix}
1 & -1 &  0 & \cdots & 0 \\
0 &  1 & -1 & \cdots & 0 \\
0 &  0 &  1 & \ddots & 0 \\
\vdots & \vdots & \vdots & \ddots & \vdots \\
0 &  0 &  0 & \cdots & 1
\end{pmatrix}
\qquad
L_1 =
\begin{pmatrix}
1 &  1 &  1 & \cdots & 1 \\
0 &  1 &  1 & \cdots & 1 \\
0 &  0 &  1 & \cdots & 1 \\
\vdots & \vdots & \vdots & \ddots & \vdots \\
0 &  0 &  0 & \cdots & 1
\end{pmatrix}
\]
In the second step we replace all but the first two columns of $U_1$
by differences (of the contained differences), yielding the matrix~$U_2$.
\begin{IEEEeqnarray*}{rCl}
U_2
 &=& ({p},
      {\Delta p},
      {\Delta^2 p},
      \translaterone{\Delta^2 p}, \dots,
      \translater{n-2}{\Delta^2 p}) \\
 &=& U_1\cdot L_2^{-1}
\end{IEEEeqnarray*}
\[
L_2^{-1} =
\begin{pmatrix}
1 &  0 &  0 & \cdots & 0 \\
0 &  1 & -1 & \cdots & 0 \\
0 &  0 &  1 & \ddots & 0 \\
\vdots & \vdots & \vdots & \ddots & \vdots \\
0 &  0 &  0 & \cdots & 1
\end{pmatrix}
\qquad
L_2 =
\begin{pmatrix}
1 &  0 &  0 & \cdots & 0 \\
0 &  1 &  1 & \cdots & 1 \\
0 &  0 &  1 & \cdots & 1 \\
\vdots & \vdots & \vdots & \ddots & \vdots \\
0 &  0 &  0 & \cdots & 1
\end{pmatrix}
\]
We repeat this procedure $n$ times, until we get
\begin{IEEEeqnarray*}{rCl}
U_n
 &=& ({p},
      {\Delta p},
      {\Delta^2 p}, \dots,
      {\Delta^n p}) \\
P &=& U_n \cdot L_n \cdot \cdots \cdot L_2 \cdot L_1
\qquad.
\end{IEEEeqnarray*}
Since the $k$-th difference of a polynomial of degree~$n$
is a polynomial of degree~$n-k$,
the matrix~$U_n$ is triangular and invertible.
Thus $P$ is invertible.
We get
\begin{IEEEeqnarray*}{rCl}
m &=& \frac{1}{2}\cdot P^{-1} \cdot \shrink{\frac{1}{2}}{p} \\
  &=& \frac{1}{2}\cdot
      L_1^{-1} \cdot L_2^{-1} \cdot \cdots \cdot L_n^{-1} \cdot
      U_n^{-1} \cdot \shrink{\frac{1}{2}}{p}
\quad,
\end{IEEEeqnarray*}
where the product $U_n^{-1} \cdot \shrink{\frac{1}{2}}{p}$
can be computed by back-substitution
and the multiplications with $L_j^{-1}$
mean computing some differences of adjacent vector elements.
\end{proof}

\begin{theorem}
\thmlabel{polynomial-to-all-masks}
If $p$ is a polynomial of degree~$n$
and $m$ is a mask that refines $\polyfunc{p}$,
then for every mask~$v$
the mask $m + v * (1,-1)^{n+1}$ refines $\polyfunc{p}$ as well.
Only masks of this kind refine $\polyfunc{p}$.
The expression $(1,-1)^{n+1}$ denotes the $(n+1)$-th convolution
of the mask $\delta_0 - \delta_1$,
that is $(1,-1)^0 = (1)$, $(1,-1)^1 = (1,-1)$, $(1,-1)^2 = (1,-2,1)$ and so on.
\end{theorem}
\newcommand\convpoly[2]{C_{#1}\cdot #2}
\newcommand\convpolypower[3]{C_{#1}^{#2}\cdot #3}
\begin{proof}
We denote the convolution of a mask~$m$ with a polynomial
by the matrix~$C_m$.
\begin{IEEEeqnarray*}{rCl}
C_m &\in& \R^{\{0,\dots,n\} \times \{0,\dots,n\}} \\
C_m &=& \sum_{i=\il}^{\ir} m_i \cdot T^i
\end{IEEEeqnarray*}
The refinement equation can be written
\[
p = 2 \cdot \shrink{2}{\convpoly{m}{p}}
\qquad.
\]
Since $C_m$ is a \Laurent{} matrix polynomial expression with respect to $T$,
it holds
\begin{IEEEeqnarray*}{rCl}
C_{h + g} &=& C_h + C_g \\
C_{h * g} &=& C_h \cdot C_g
\qquad.
\end{IEEEeqnarray*}

\pagebreak[3]
\begin{IEEEeqnarray*}{rCl"s}
\noalign{$2 \cdot \shrink{2}{\convpoly{m + v * (1,-1)^{n+1}}{p}}$}
 &=& \shrink{2}{(2 \cdot \convpoly{m}{p})} +
     \shrink{2}{(2 \cdot \convpoly{v * (1,-1)^{n+1}}{p})} \\
 &=& p + \shrink{2}{(2 \cdot \convpoly{v}{(\convpolypower{(1,-1)}{n+1}{p})})} \\
\noalign{$(n+1)$-th difference of an $n$-degree polynomial vanishes:
   $\convpolypower{(1,-1)}{n+1}{p} = 0$}
 &=& p
\end{IEEEeqnarray*}

We still have to show,
that refining masks of $\polyfunc{p}$
always have the form $m + v * (1,-1)^{n+1}$.
Consider a mask~$h_1$ that refines $\polyfunc{p}$.
By computing the \Laurent{} polynomial division remainder
with respect to the divisor $(1,-1)^{n+1}$
we can reduce $h_1$ to a mask~$h_0$ that has support in $\{0,\dots,n\}$
and we can reduce $m$ to a mask~$g$ with support in $\{0,\dots,n\}$, too.
\begin{IEEEeqnarray*}{rCl}
m   &=& g   + v_0 * (1,-1)^{n+1} \\
h_1 &=& h_0 + v_1 * (1,-1)^{n+1}
\end{IEEEeqnarray*}
From the above considerations we conclude
that both $g$ and $h_0$ refine $\polyfunc{p}$
and the uniqueness property in \thmref{polynomial-to-mask}
eventually gives us $g = h_0$, thus
\[ h_1 = m + (v_1-v_0) * (1,-1)^{n+1} \qquad.\]

\end{proof}

\begin{remark}
By adding terms of the form $v * (1,-1)^{n+1}$
we can shift the support of a mask
and still refine the same polynomial.
\end{remark}

\subsection{Example}

For better comprehension of the theorems of the previous section
let us examine a longer example.
We would like to illustrate that the same refinement mask
can refine different functions.
However, if we restrict the function class to,
say, continuous compactly supported functions or to polynomial functions,
then a refinement mask is associated with a unique function.
We would like to compare a continuous compactly supported function
with a polynomial function,
both being refinable with respect to the same mask.
Unfortunately this is not possible,
since for the former type of functions we need masks with sum~$1$,
whereas for polynomial functions we need masks with sums
that are powers of two that are smaller than $1$.
So, we are going to compare
antiderivatives of the quadratic B-spline
and polynomial functions that are refinable with respect to the masks
$
\frac{1}{16}\cdot(1,3,3,1),
\frac{1}{32}\cdot(1,3,3,1),
\frac{1}{64}\cdot(1,3,3,1)
$ according to \lemref{antiderivative}.

\begin{figure}
\begin{tabular}{rr}
\begin{tikzpicture}[domain=0:3,xscale=1.5,yscale=3/1.2]
  \drawframe{-0.2}{3.2}{0,0.5,...,3} {-0.1}{1.1}{0.0,0.2,0.4,0.6,0.8,1.0}
  \draw plot file {figures/RefinementBSpline1Int.csv};
  \draw[dash pattern=on 2pt off 1pt] plot file {figures/RefinementBSpline1Int0.csv};
  \draw[dash pattern=on 1pt off 1pt] plot file {figures/RefinementBSpline1Int1.csv};
  \draw[dash pattern=on 2pt off 1pt] plot file {figures/RefinementBSpline1Int2.csv};
  \draw[dash pattern=on 1pt off 1pt] plot file {figures/RefinementBSpline1Int3.csv};
\end{tikzpicture}
&
\begin{tikzpicture}[domain=0:3,xscale=1.5,yscale=3/1.2]
  \drawframe{-0.2}{3.2}{0,0.5,...,3} {-0.1}{1.1}{0.0,0.2,0.4,0.6,0.8,1.0}
  \draw plot file {figures/Refinement0Polynomial.csv};
  \draw[dash pattern=on 2pt off 1pt] plot file {figures/Refinement0Polynomial0.csv};
  \draw[dash pattern=on 1pt off 1pt] plot file {figures/Refinement0Polynomial1.csv};
  \draw[dash pattern=on 2pt off 1pt] plot file {figures/Refinement0Polynomial2.csv};
  \draw[dash pattern=on 1pt off 1pt] plot file {figures/Refinement0Polynomial3.csv};
\end{tikzpicture}
\\
\begin{tikzpicture}[domain=0:3,xscale=1.5,yscale=3/1.8]
  \drawframe{-0.2}{3.2}{0,0.5,...,3} {-0.1}{1.7}{0.0,0.5,1.0,1.5}
  \draw plot file {figures/RefinementBSpline2Int.csv};
  \draw[dash pattern=on 2pt off 1pt] plot file {figures/RefinementBSpline2Int0.csv};
  \draw[dash pattern=on 1pt off 1pt] plot file {figures/RefinementBSpline2Int1.csv};
  \draw[dash pattern=on 2pt off 1pt] plot file {figures/RefinementBSpline2Int2.csv};
  \draw[dash pattern=on 1pt off 1pt] plot file {figures/RefinementBSpline2Int3.csv};
\end{tikzpicture}
&
\begin{tikzpicture}[domain=0:3,xscale=1.5,yscale=3/3.4]
  \drawframe{-0.2}{3.2}{0,0.5,...,3} {-1.7}{1.7}{-1,0,1}
  \draw plot file {figures/Refinement1Polynomial.csv};
  \draw[dash pattern=on 2pt off 1pt] plot file {figures/Refinement1Polynomial0.csv};
  \draw[dash pattern=on 1pt off 1pt] plot file {figures/Refinement1Polynomial1.csv};
  \draw[dash pattern=on 2pt off 1pt] plot file {figures/Refinement1Polynomial2.csv};
  \draw[dash pattern=on 1pt off 1pt] plot file {figures/Refinement1Polynomial3.csv};
\end{tikzpicture}
\\
\begin{tikzpicture}[domain=0:3,xscale=1.5,yscale=3/1.8]
  \drawframe{-0.2}{3.2}{0,0.5,...,3} {-0.1}{1.7}{0.0,0.5,1.0,1.5}
  \draw plot file {figures/RefinementBSpline3Int.csv};
  \draw[dash pattern=on 2pt off 1pt] plot file {figures/RefinementBSpline3Int0.csv};
  \draw[dash pattern=on 1pt off 1pt] plot file {figures/RefinementBSpline3Int1.csv};
  \draw[dash pattern=on 2pt off 1pt] plot file {figures/RefinementBSpline3Int2.csv};
  \draw[dash pattern=on 1pt off 1pt] plot file {figures/RefinementBSpline3Int3.csv};
\end{tikzpicture}
&
\begin{tikzpicture}[domain=0:3,xscale=1.5,yscale=3/2.8]
  \drawframe{-0.2}{3.2}{0,0.5,...,3} {-0.1}{2.7}{0,1,2}
  \draw plot file {figures/Refinement2Polynomial.csv};
  \draw[dash pattern=on 2pt off 1pt] plot file {figures/Refinement2Polynomial0.csv};
  \draw[dash pattern=on 1pt off 1pt] plot file {figures/Refinement2Polynomial1.csv};
  \draw[dash pattern=on 2pt off 1pt] plot file {figures/Refinement2Polynomial2.csv};
  \draw[dash pattern=on 1pt off 1pt] plot file {figures/Refinement2Polynomial3.csv};
\end{tikzpicture}
\end{tabular}
\caption{%
Refinement of antiderivatives of the quadratic B-spline (left column)
and polynomial functions (right column)
with respect to the masks $
\frac{1}{16}\cdot(1,3,3,1),
\frac{1}{32}\cdot(1,3,3,1),
\frac{1}{64}\cdot(1,3,3,1)$
(first to last row).
The meaning of the lines is the same as in \figref{refinement}.
}
\figlabel{polynomial-refinement}
\end{figure}
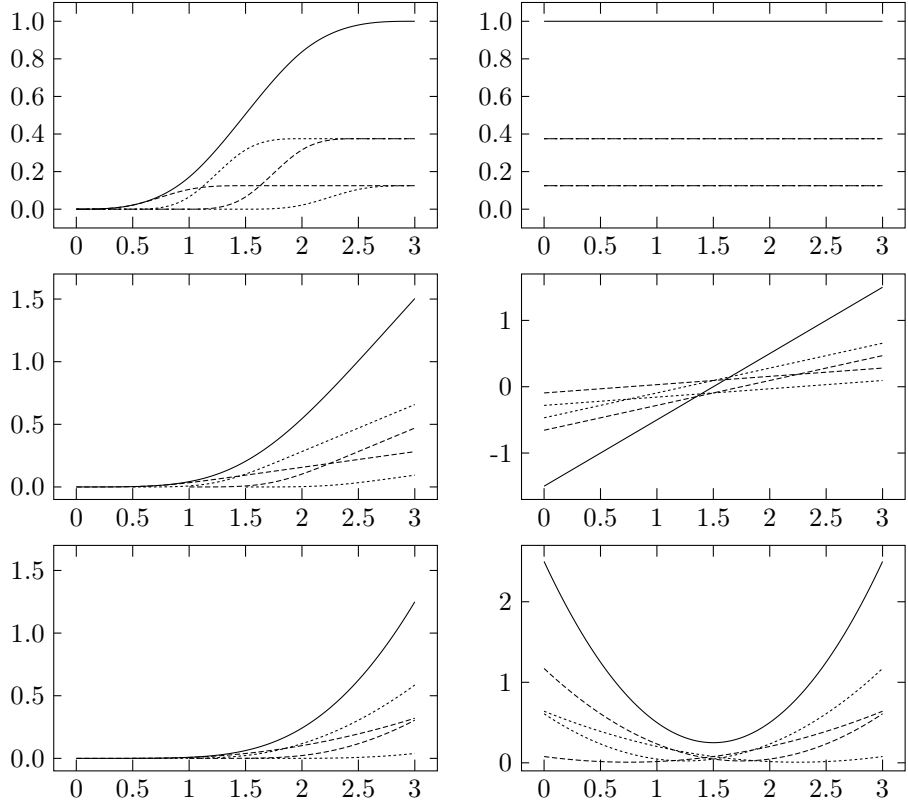
\figref{polynomial-refinement} shows the refinable functions.
Since the quadratic B-spline as in \figref{refinement}
is refinable with respect to the mask $\frac{1}{8}\cdot(1,3,3,1)$,
its antiderivative as in the top-left plot in \figref{polynomial-refinement}
is refined by $\frac{1}{16}\cdot(1,3,3,1)$.
The sum of this mask is $\frac{1}{2}$
and thus the only refinable polynomial for that mask is a constant polynomial.
We normalize it to $1$.
This is shown in the top-right plot of \figref{polynomial-refinement}.

The second row is associated with mask $\frac{1}{32}\cdot(1,3,3,1)$.
We get the refined polynomial by integrating the function $t \mapsto 1$
as in the proof of \thmref{mask-to-polynomial}:
\begin{IEEEeqnarray*}{rCl}
\polyfunc{Q}(t) &=& t \\
p_1 &=& 1 \qquad
p_0 = \frac{2}{1/2} \cdot \frac{1}{32}\cdot(1\cdot0 + 3\cdot(-1) + 3\cdot(-2) + 1\cdot(-3))
    = -\frac{3}{2} \\
\polyfunc{p}(t) &=& -\frac{3}{2} + t
\end{IEEEeqnarray*}
That is, $t \mapsto -\frac{3}{2} + t$
is refined by $\frac{1}{32}\cdot(1,3,3,1)$.
We repeat this procedure in order to get to the last row
of \figref{polynomial-refinement}.
We start with the antiderivative of $t \mapsto -\frac{3}{2} + t$:
\begin{IEEEeqnarray*}{rCl}
\polyfunc{Q}(t) &=& -\frac{3\cdot t}{2} + \frac{t^2}{2} \\
p_2 &=& 1 \qquad
p_1 = -3 \qquad
p_0 = \frac{2}{1/2\cdot 3/4} \cdot
      \frac{1}{64}\cdot(1\cdot0 + 3\cdot 2 + 3\cdot 5 + 1\cdot 9)
    = \frac{5}{2} \\
\polyfunc{p}(t) &=& \frac{5}{2} - 3\cdot t + t^2
\end{IEEEeqnarray*}

We want to check whether the obtained polynomial function
$t \mapsto \frac{5}{2} - 3\cdot t + t^2$
is actually refined by $\frac{1}{64}\cdot(1,3,3,1)$
according to \eqnref{refinement}:
\begin{IEEEeqnarray*}{rCl"rCl}
\polyfunc{p}(2t-0)   &=& \frac{ 5}{2} -  6\cdot t + 4\cdot t^2
&
1\cdot\polyfunc{p}(2t-0)   &=& \frac{ 5}{2} -  6\cdot t + 4\cdot t^2
\\
\polyfunc{p}(2t-1) &=& \frac{13}{2} - 10\cdot t + 4\cdot t^2
&
3\cdot\polyfunc{p}(2t-1) &=& \frac{39}{2} - 30\cdot t + 12\cdot t^2
\\
\polyfunc{p}(2t-2) &=& \frac{25}{2} - 14\cdot t + 4\cdot t^2
&
3\cdot\polyfunc{p}(2t-2) &=& \frac{75}{2} - 42\cdot t + 12\cdot t^2
\\
\polyfunc{p}(2t-3) &=& \frac{41}{2} - 18\cdot t + 4\cdot t^2
&
1\cdot\polyfunc{p}(2t-3) &=& \frac{41}{2} - 18\cdot t + 4\cdot t^2
\\
\hline
&& & && 80 - 96\cdot t + 32\cdot t^2
\end{IEEEeqnarray*}
\[
\frac{2}{64}\cdot(80 - 96\cdot t + 32\cdot t^2)
 = \frac{5}{2} - 3\cdot t + t^2
\]
That is, the mask actually refines the polynomial function.

In the next step of our example
we want to determine refinement masks for our polynomial functions
by means of \thmref{polynomial-to-mask}.
However, it is already clear
that we will not get the mask $\frac{1}{64}\cdot(1,3,3,1)$ as a result
because \thmref{polynomial-to-mask} only promises
a mask of size~3 for a quadratic polynomial.
Nonetheless this smaller mask should be compatible to the original one
in the sense of \thmref{polynomial-to-all-masks}.
\begin{IEEEeqnarray*}{rCl}
\polyfunc{p}(t)
 &=& \frac{5}{2} - 3\cdot t + t^2 \\
\polyfunc{\Delta p}(t)
 &=& \polyfunc{p}(t-1) - \polyfunc{p}(t) \\
 &=& \left(\frac{5}{2} - 3\cdot(t-1) + (t-1)^2\right)
   - \left(\frac{5}{2} - 3\cdot t + t^2\right) \\
 &=& 4 - 2\cdot t \\
\polyfunc{\Delta^2 p}(t)
 &=& \polyfunc{\Delta p}(t-1) - \polyfunc{\Delta p}(t) \\
 &=& \left(4 - 2\cdot (t-1)\right)
   - \left(4 - 2\cdot t\right) \\
 &=& 2
\end{IEEEeqnarray*}
\begin{IEEEeqnarray*}{rCl}
U_2 &=&
\left(p, \Delta p, \Delta^2 p \right)
=
\begin{pmatrix}
\frac{5}{2} & 4 & 2 \\
-3 & -2 & 0 \\
1 & 0 & 0
\end{pmatrix}
\\
m &=& \frac{1}{2} \cdot
      L_1^{-1} \cdot L_2^{-1} \cdot U_2^{-1} \cdot \shrink{\frac{1}{2}}{p}
   =  L_1^{-1} \cdot L_2^{-1} \cdot U_2^{-1} \cdot
      \frac{1}{8} \cdot \begin{pmatrix} 10 \\ -6 \\ 1 \end{pmatrix} \\
  &=& L_1^{-1} \cdot L_2^{-1} \cdot
      \frac{1}{32} \cdot \begin{pmatrix} 4 \\ 6 \\ 3 \end{pmatrix}
   =  L_1^{-1} \cdot
      \frac{1}{32} \cdot \begin{pmatrix} 4 \\ 3 \\ 3 \end{pmatrix} \\
  &=& \frac{1}{32} \cdot \begin{pmatrix} 1 \\ 0 \\ 3 \end{pmatrix} \\
\end{IEEEeqnarray*}
We leave it to the reader to verify
that this mask actually refines our polynomial.

The last thing we want to check is,
that the difference between the short mask and the original mask
is a convolutional multiple of $(1,-1)^3$
as stated by \thmref{polynomial-to-all-masks}:
\[
\frac{1}{64}\cdot(1,3,3,1) - \frac{1}{32}\cdot(1,0,3,0)
 = \frac{1}{64}\cdot(-1,3,-3,1)
 = -\frac{1}{64}\cdot(1,-1)^3
\quad.
\]

\subsection{The cascade algorithm}

\seclabel{cascade-algorithm}

We want to close the section on the theoretical results
with an alternative way to compute the polynomial that is refined by a mask.
It is an iterative algorithm
for the approximate computation of a polynomial,
and it is analogous to the \indexkeyword{cascade algorithm}
known for refinable functions of bounded support.%
\cite{strang1996cascade}
The refinement relation
\begin{IEEEeqnarray*}{rCl}
{p} &=& 2 \cdot \shrink{2}{C_m \cdot p}
\end{IEEEeqnarray*}
is interpreted as a recursively defined function sequence with
\begin{IEEEeqnarray*}{rCl}
{p_{j+1}} &=& 2 \cdot \shrink{2}{C_m \cdot p_j}
\quad.
\end{IEEEeqnarray*}
This iteration is in fact the vector iteration method
for computing the eigenvector that corresponds to the largest eigenvalue.
\begin{theorem}
\thmlabel{cascade-algorithm}
Given a mask~$m$ that sums up to $2^{-n-1}$ for a given natural number~$n$
and a starting polynomial~$p_0$ of degree~$n$
that is not orthogonal to the refined polynomial,
the recursion
\[ {p_{j+1}} = 2 \cdot \shrink{2}{C_m \cdot p_j} \]
converges and the limit polynomial $\lim_{j\to\infty} p_j$ is refined by $m$.

An appropriate choice for $p_0$ is $(0,\dots,0,1)$.
\end{theorem}
\begin{proof}
The matrix~$\shrink{2}{C_m}$ expands to
\begin{IEEEeqnarray*}{rCl}
\shrink{2}{C_m} &=&
\begin{cases}
  2^{j} \cdot \binom{k}{j} \cdot
     \sum_{i=\il}^{\ir} (-i)^{k-j} \cdot m_i
    &: j \le k \\
  0 &: j > k
\end{cases}
\end{IEEEeqnarray*}
and thus is of upper triangular shape.
This implies that the diagonal elements $2^j \cdot \sum_{i=\il}^{\ir} m_i$
for $j\in\{0,\dots,n\}$ are the eigenvalues.
Because the mask sums up to $2^{-n-1}$
the eigenvalues of $2\cdot \shrink{2}{C_m}$ are $\{1,\dots,2^{-n}\}$.
That is, the largest eigenvalue is 1 and it is isolated.
These are the conditions for the vector iteration method,
consequently the iteration converges to a vector
that is a fixed point of the refinement operation.
\end{proof}

\begin{remark}
The eigenvectors of the eigenvalues are the respective derivatives
of the main refinable polynomial.
\end{remark}

\section{Implementation}

The presented conversions from masks to polynomials and back
are implemented in the functional programming language Haskell
as module \texttt{MathObj.RefinementMask2}
of the NumericPrelude project\cite{thurston2010numeric-prelude-0.2}.
However, note that the definition of the Haskell functions
slightly differ from this paper,
since the factor $2$ must be part of the mask.

\section{Related work}

So far, refinable functions were mostly explored
in the context of wavelet theory.
In this context an important problem
was to design refinement masks that lead
to smooth finitely supported refinable functions.%
\cite{villemoes1993sobolev,villemoes1994regularity}
It was shown that smoothness can be estimated the following way:
Decompose the refinement mask~$m$ into the form $(1,1)^n * v$,
where $n$ is chosen maximally.
According to \lemref{refinable-convolution} this corresponds
to a convolution of functions.
However, strictly speaking, $v$ corresponds to a distribution.
The exponent $n$ represents the order of a B-spline
and is responsible for the smoothness of the refinable function,
whereas for $v$ there is an eigenvalue problem,
where the largest eigenvalue determines
how much the smoothness of the B-spline is reduced.

The cascade algorithm\cite{strang1996cascade}
was developed in order to compute numerical approximations
to refinable functions.
A combination of the cascade algorithm and \lemref{refinable-convolution}
was used by \refcite{villiers2000waveletinteger}
for computing scalar products and other integrals of products
of refinable functions.
This is required for solving partial differential equations
using a wavelet \person{\foreignlanguage{russian}{Gale0rkin}} approach.

Discrete wavelet functions in a multiresolution analysis
are defined in terms of refinable functions,
but were not considered refinable functions at first.
However in \refcite{strang1999infinitemask} it is shown,
that wavelets are refinable with respect to infinite masks.
The trick is to use polynomial division
for dividing the wavelet masks of adjacent scales:
If $\varphi$ is refinable with respect to mask~$h$,
and $\psi$ is a wavelet with respect to mask~$g$ and generator $\varphi$,
then $\psi$ is refinable with respect to $\frac{g \uparrow 2}{g}\cdot h$,
where $g \uparrow 2$ is $g$ upsampled by a factor of 2.

Although polynomial functions are not in the main focus
of the research on refinable functions,
we got to know several discoveries of this relation
after uploading our work to arXiv.
The thesis~\refcite{malone2000dilation}
is the first reference known to us
that explains how to obtain a polynomial function that is refined by a mask.
The approach in this thesis is also based on polynomial differentiation.

The first source known to us that describes the opposite way,
i.e. how to find a refining mask for a polynomial function,
is \refcite{king2011refinablepolynomial}.%
\footnote{It is written in 2011, but was already presented in talks in 2003.}
In terms of our matrices the author of that article
does not just use the matrix~$P$ of shifted polynomials,
but its factorization $P = A\cdot V$.
The matrix~$V$ is a \person{Vandermonde} matrix with
\begin{IEEEeqnarray*}{rCl}
V &\in& \R^{\{0,\dots,n\}\times\{0,\dots,n\}} \\
V_{i,j} &=& j^i
\end{IEEEeqnarray*}
The matrix~$A$ is defined using the derivatives of $p$
by $A = \left(\frac{p}{0!}, \frac{p'}{1!}, \dots, \frac{p^{(n)}}{n!}\right)$,
or element-wise by
\begin{IEEEeqnarray*}{rCl}
A &\in& \R^{\{0,\dots,n\}\times\{0,\dots,n\}} \\
A_{i,j} &=&
\begin{cases}
{i+j \choose i} \cdot p_{i+j} &: i+j \le n \\
0 &: i+j > n
\end{cases}
.
\end{IEEEeqnarray*}
The \person{Taylor} expansion of $p$ allows to express translations of $p$
in terms of its derivatives:
\[
\polyfunc{p}(t+x) =
   \frac{\polyfunc{p}(t)}{0!} + x\cdot \frac{\polyfunc{p}'(t)}{1!}
    + \dots + x^{n}\cdot \frac{\polyfunc{p}^{(n)}(t)}{n!}
.
\]
Thus multiplying $A$ and $V$ yields the sequence of shifted polynomials
with $x \in \{0, -1, \dots, -n\}$.

With this factorization the invertibility of $P$ follows obviously
from the invertibility of the triangular matrix~$A$ and
the invertibility of the \person{Vandermonde} matrix~$V$
with respect to pairwise distinct nodes.

Actually, in the article the general \person{Vandermonde} matrix
with pairwise distinct nodes $\{\ell_0, \dots, \ell_n\}$ is used:
\[
V=
\begin{pmatrix}
1 & 1 & \cdots & 1 \\
-\ell_0 & -\ell_1 & \cdots & -\ell_n \\
\vdots & \vdots & \ddots & \vdots \\
(-\ell_0)^ n & (-\ell_1)^ n & \cdots & (-\ell_n)^ n
\end{pmatrix}
\]
That is, a refinement mask for a polynomial can also be found
when only a certain set of $n+1$ nodes is allowed to be non-zero.
The nodes may even be non-integral.

In retrospect we could have conducted our proof
of \thmref{polynomial-to-mask} with arbitrary nodes, too.
We would have to define $P = (T^{\ell_0}\cdot p, \dots, T^{\ell_n}\cdot p)$
and then use divided differences instead of simple differences
for the LU decomposition.

The paper~\refcite{gustafson2006refine} extends the previous one
by an exploration of the refinability of rational functions.
They find that a rational function~$\varphi$ is refinable if and only if
there is a real sequence $s$
($s\in\summablesequences{0}{\Z}$, i.e.\ a \Laurent{} polynomial)
and a positive natural number $k$ such that
\[
\varphi(t) = \sum_{i\in\mathbb{Z}} \frac{s_i}{(t-i)^k}
\qquad
\land
\qquad
s | (s \uparrow 2) ,
\]
that is, $\exists \polyfunc{m}(z)\in\R[z,z^{-1}]\ \polyfunc{m}(z)\cdot \polyfunc{s}(z) = \polyfunc{s}(z^2)$,
where $2^{k-1}\cdot m$ is the refinement mask.

\section{Future work}

There are some obvious generalizations to be explored:
refinement with respect to factors different from 2,
separable multidimensional refinement
and most general multidimensional refinement
with respect to arbitrary dilation matrices.

Another interesting question is the following one:
By \lemref{refinable-convolution} we know,
that convolution of functions maps to convolution of their refinement masks.
We can use this for defining a kind of convolution.
In order to convolve two functions $\varphi_0$ and $\varphi_1$,
we compute refining masks $m_0$ and $m_1$, respectively,
convolve the masks and then find a function that is refined by $m_0 * m_1$.
In case of polynomial functions there is no notion of convolution
because the involved integrals diverge.
We can however define a convolution based on refinement.
Unfortunately, the mapping from a polynomial function
to a refinement mask is not unique,
consequently the defined convolution is not unique as well --
not to speak of the arbitrary constant factor.
If we choose arbitrary masks from the admissible ones,
then the convolution is not distributive with addition,
i.e. $\psi * (\varphi_0 + \varphi_1)
         = \psi * \varphi_0 + \psi * \varphi_1$
is not generally satisfied.
The open question is, whether it is possible to choose masks for polynomials,
such that the polynomial convolution via refinement
is commutative, associative and distributive.

\section{Acknowledgment}

I like to thank David Larson and David Malone for pointing me to related research
and Emily King for careful proof-reading and discussion of alternative proofs.

\bibliography{wavelet,literature,thielemann,hackage}
\bibliographystyle{plain}

\end{document}